\documentclass[11pt]{amsart}

\usepackage{color}
\usepackage{epsfig}
\usepackage{psfrag, graphics}
\usepackage{amsthm}
\usepackage{amssymb}
\usepackage{amsmath}
\usepackage{amscd}
\usepackage{graphicx}
\usepackage{epstopdf}
\usepackage{stmaryrd}

\newcommand{\eqdef}{\;{:=}\;}

\newtheorem {Theorem}   {Theorem}

\numberwithin{Theorem}{section}

\newtheorem {Lemma}[Theorem]    {Lemma}
\newtheorem {Proposition}[Theorem]{Proposition}

\theoremstyle{definition}
\newtheorem{Definition}[Theorem]{Definition}
\theoremstyle{remark}
\newtheorem{Remark}[Theorem]{Remark}
\newtheorem{Example}[Theorem]{Example}

\expandafter\chardef\csname pre amssym.def
at\endcsname=\the\catcode`\@ \catcode`\@=11
\def\undefine#1{\let#1\undefined}
\def\newsymbol#1#2#3#4#5{\let\next@\relax
 \ifnum#2=\@ne\let\next@\msafam@\else
 \ifnum#2=\tw@\let\next@\msbfam@\fi\fi
 \mathchardef#1="#3\next@#4#5}
\def\mathhexbox@#1#2#3{\relax
 \ifmmode\mathpalette{}{\m@th\mathchar"#1#2#3}%
 \else\leavevmode\hbox{$\m@th\mathchar"#1#2#3$}\fi}
\def\hexnumber@#1{\ifcase#1 0\or 1\or 2\or 3\or 4\or 5\or 6\or 7\or 8\or
 9\or A\or B\or C\or D\or E\or F\fi}

\font\teneufm=eufm10 \font\seveneufm=eufm7 \font\fiveeufm=eufm5
\newfam\eufmfam
\textfont\eufmfam=\teneufm \scriptfont\eufmfam=\seveneufm
\scriptscriptfont\eufmfam=\fiveeufm

\catcode`\@=\csname pre amssym.def at\endcsname

\newcommand{\AC}{{\mathcal A}}

\newcommand{\EE}{{\mathcal E}}
\newcommand{\FF}{{\mathcal F}}

\newcommand{\HH}{{\mathcal H}}

\newcommand{\MM}{{\mathcal M}}

\newcommand{\PP}{{\mathcal P}}

\newcommand{\SC}{{\mathcal S}}

\newcommand{\WW}{{\mathcal W}}

\def    \C      {{\mathbb C}}
\def    \R      {{\mathbb R}}
\def    \Z      {{\mathbb Z}}

\def    \T      {{\mathbb T}}

\def    \vvv   {|\hspace{-0.1em}|\hspace{-0.1em}|}
\def    \eqdef    {\;{:=}\;}
\def    \id      {\operatorname{id}}

\def    \om       {\omega}
\def    \eps      {\epsilon}

\def    \p        {\partial}
\def    \d     {\delta}

\def    \Crit      {\operatorname{Crit}}

\def    \12      {\frac{1}{2}}
\def    \Vol      {\operatorname{Vol}}
\def    \hl       {H_{\scriptscriptstyle{L}}}
\def    \hlo       {\widetilde{H}_{\scriptscriptstyle{L}}}

\def    \H   {\operatorname{H}}  
\def    \HF   {\operatorname{HF}} 
\def    \CF   {\operatorname{CF}} 
\def    \CM   {\operatorname{CM}} 
\def    \cz   {\operatorname{\mu_{\scriptscriptstyle{CZ}}}} 
 
\def    \masl   {\operatorname{\mu^{\scriptscriptstyle{L}}_{\scriptscriptstyle{Maslov}}}}
\def    \mor   {\operatorname{I_{\scriptscriptstyle{Morse}}}} 

\begin{document} 


\title[Action selectors and Maslov class rigidity]{Action selectors and Maslov class rigidity}

\author[Ely Kerman]{Ely Kerman}
\address{Department of Mathematics, University of Illinois at Urbana-Champaign, Urbana, IL 61801, USA }
\email{ekerman@math.uiuc.edu}

\date{\today}


\bigskip



\begin{abstract} 
In this paper we detect new restrictions on the Maslov class
of displaceable Lagrangian submanifolds of symplectic manifolds which are symplectically aspherical. 
These restrictions are established using action selectors for Hamiltonian flows. In particular, we construct and utilize a new  action selector for the flows of a special  class of 
Hamiltonian functions which arises naturally in the study of Hamiltonian 
paths which minimize the Hofer length functional.  
\end{abstract} 


\maketitle

\section{INTRODUCTION}

One of the central rigidity phenomena in symplectic topology concerns restrictions on the Maslov class of a
Lagrangian submanifold which can not be detected by topological methods alone. 
This type of rigidity was discovered independently by Viterbo in \cite{vi} and Polterovich in \cite{po:a}, for  Lagrangian submanifolds of symplectic vector spaces
and certain cotangent bundles.  
In this paper we study 
Maslov class rigidity for Lagragian submanifolds of closed symplectic manifolds  which are 
{\it symplectically aspherical}. That is, we assume that  both the symplectic form
and the first Chern class vanish on the second homotopy group. Such symplectic manifolds
hold a special place in symplectic topology,  in particular for results which concern
the contractible periodic orbits of Hamiltonian flows, see for example \cite{cz, gi:con, hi, sz, schw}.  In this setting, Hamiltonian Floer homology
becomes a more precise tool for studying periodic orbits, in part,  because the actions and Conley-Zehnder indices of these orbits are both well-defined.  
We will exploit this precision to establish new restrictions on the Maslov class of displaceable Lagrangian submanifolds, i.e., Lagrangian submanifolds which can be displaced by a 
Hamiltonian diffeomorophism.  
In the process, we will construct  a new action 
selector  for Hamiltonian functions which have properties that are necessary to generate 
Hamiltonian paths that  minimize the (positive) Hofer length, but fail to do so.

Let $L$ be a Lagrangian submanifold of a symplectic manifold 
$(M, \om)$.  Recall that the Maslov class of  $L$  is a homomorphism  $\masl$ from the second relative 
homotopy group $\pi_2(M,L)$ to  $\Z$, and the minimial Maslov
number of $L$ is the largest positive integer $N_L$ such that the  image of $\masl$ is contained in $N_L \Z$. When
$(M, \om)$ is symplectically aspherical, more precisely when the first Chern class of $M$  vanishes 
on $\pi_2(M)$, the  Maslov class descends to a homomorphism
\begin{equation*}
\masl \colon \ker i^* \to \Z, 
\end{equation*}
where $\ker i^*$ is  the kernel of the map $i^* \colon \pi_1(L) \to \pi_1(M)$ induced by the inclusion $i \colon L \hookrightarrow M$. 
In the aspherical setting, the symplectic form also yields a homomorphism $\omega \colon \ker i^* \to \R$ defined by integrating $\omega$ over spanning discs.

As in  \cite{vi},  we will study $\masl$ by considering 
Hamiltonian flows which are supported in normal neighborhoods of $L$, where they can be identified with reparameterizations of 
geodesic flows on $T^*L$. For displaceable Lagrangian submanifolds, these Hamiltonian flows will  eventually fail to minimize the Hofer length functional, 
and this failure will  be shown to imply  the existence of special periodic orbits via an action selector which is defined in terms of Hamiltonian Floer theory. The Conley-Zehnder indices  of these special periodic orbits
will yield  bounds on the Maslov indices  of the corresponding  closed geodesics on $L$, and these bounds will imply the desired rigidity results.

Let $g$ be a Riemannian metric on $L$. The closed geodesics of $g$, with period equal to one, are the critical points of the 
energy functional $\EE_g$ which is defined on the space of smooth loops  $ C^{\infty}(\R/\Z, L)$ by
$$
\EE_g(q(t)) = \int_0^1 \12  |\dot{q}(t)|^2   \, dt.
$$
For a Morse-Bott nondegenerate energy functional $\EE_g$, we will denote by $d_g$ 
the maximum dimension of a critical submanifold of $\EE_g$ which consists of nonconstant 
closed geodesics. The nullspace of the Hessian of $\EE_g$ at a closed geodesic $q(t)$ is spanned by the periodic Jacobi fields along $q(t)$.
If the  metric $g$ has nonpositive sectional, then  these Jacobi fields must be constant  and so $d_g \leq \dim L$. For example,
for  the flat metric on a torus one has $d_g =\dim L$, and if $g$ is a metric with negative sectional curvature then $d_g=1$.

The following is our main result.

\begin{Theorem}
\label{thm1}
Let $L$ be a displaceable Lagrangian submanifold of  a closed symplectic manifold 
$(M, \om)$ which has  dimension $2n$  and is symplectically aspherical.  Suppose that $L$ admits a metric $g$ of nonpositive sectional curvature whose energy functional $\EE_g$ is Morse-Bott. Then there is a  class $A$  in $\ker i^*$ with $\omega(A)>0$ and 
$$\masl(A) \in [n+1-d_g, n+1 \,(+1)],$$
where the $(+1)$ term does not contribute if $L$ is orientable.  In particular, since $d_g \leq n$, we have $N_L \in [1, n+1 \,(+1)]$.
\end{Theorem}

The bounds of Theorem \ref{thm1} can often be further refined by using  the fact that 
the Maslov class of an orientable Lagrangian submanifold  takes values in $2\Z$. For example,
in dimension four we have

\begin{Theorem}
\label{surface}
Let $L$ be a displaceable Lagrangian surface of  a closed four-dimensional symplectic manifold 
$(M, \om)$ which is symplectically aspherical. If $L$ is orientable and has genus at least one, then
$N_L=2$.
\end{Theorem}

For the important case of Lagrangian tori we obtain the following result.

\begin{Theorem}
\label{torus}
Let $(M, \om)$ be  a closed symplectic manifold 
of dimension $2n$  which is symplectically aspherical.
If $L$ is  a displaceable Lagrangian torus in 
$(M, \om)$, then $N_L \in [2,\, n+1].$
\end{Theorem}

\noindent This generalizes the main theorem of \cite{vi}, which holds for Lagrangian tori in symplectic vector spaces.
By  Darboux's theorem, Theorem \ref{torus} also implies Viterbo's result. On the other hand, one can easily construct examples of displaceable Lagrangian tori 
which are not contained in any symplectically embedded ball.

\medskip

\noindent{\bf Lagrangian intersections.} Theorem \ref{thm1}  can also be used to prove that certain Lagrangian submanifolds can not be displaced by a Hamiltonian diffeomorphism.  Consider the following generalization 
of Theorem 3.3.3 from Chapter X of \cite{aula}.
Let  $\theta_1, \dots, \theta_n$ be coordinates on the $n$-torus, $\T^n = \R/\Z \times \dots \times \R/\Z$, and let $G$ be the group of transformations on $\T^n$ generated by the 
map 
$$
(\theta_1, \dots, \theta_n) \mapsto (\theta_1+(2n-2)^{-1},\theta_3 \dots, \theta_n, -\theta_2).
$$
Define $\mathbb{K}^n$ to be the quotient $\T^n/G.$

\begin{Theorem}
\label{klein}
Let $(M, \om)$ be  a closed symplectic manifold 
of dimension $2n \geq 6$ which is symplectically aspherical.
If $L$ is  a  Lagrangian submanifold of
$(M, \om)$ which is diffeomorphic to $\mathbb{K}^n$, then $L$ can not be displaced 
by any Hamiltonian diffeomorphism.
\end{Theorem}

\begin{proof}
Assume that $L$ is displaceable.
For the flat metric on $\mathbb{K}^n$ we have $d_g = n$ and Theorem \ref{thm1}  implies that 
$N_L$ is in $[1, n+1\, (+1)]$.   Let $U$ be a neighborhood of $L$ in $M$  which is small enough to be  both displaceable and symplectomorphic to a neighborhood
of the zero section in $T^*\mathbb{K}^n$. As described in \cite{aula}, pages 296--297, there is a Lagrangian torus 
$L'$ in  $U$ such that 
\begin{equation}
\label{dilate}
N_{L'} =(2n-2)N_L.
\end{equation} 
The Lagrangian torus $L'$ is also  displaceable and, 
for $n>3$, equation \eqref{dilate} contradicts Theorem \ref{torus} since $N_{L'} = (2n-2)N_L > n+1$.  For $n=3$, $\mathbb{K}^3$ is orientable
and so $N_{L} \geq 2$. In this case, equation \eqref{dilate} implies that  $N_{L'} \geq 6$ which again contradicts Theorem \ref{torus}. Hence, $L$ is not displaceable.
\end{proof}

\medskip

\noindent{\bf Split hyperbolic Lagrangian submanifolds.} 
Following \cite{ks},  we say that a product manifold  $L=P_1 \times \dots \times P_k$ is   {\bf split hyperbolic} if each factor $P_j$ admits a metric with negative sectional 
curvature. By convention, we label the factors of $L$  so that $\dim P_j \leq \dim P_{j+1}$. 
As noted in \cite{ks}, every symplectic manifold with dimension different from two or six, contains Lagrangian submanifolds 
which are both split hyperbolic and displaceable.  For this large class of Lagrangians, Theorem \ref{thm1} yields the following rigidity result.

\begin{Theorem}
\label{splits}
Let $L=P_1 \times \dots \times P_k$ be a split hyperbolic Lagrangian submanifold of  a closed symplectic manifold $(M,\om)$ of dimension $2n$
which is symplectically aspherical.  If $L$ is displaceable, then there is a class $A \in \ker i^*$ such that $\masl(A) \in [\dim P_1, n+1\,(+1)].$
\end{Theorem}

\begin{proof}
For each factor $P_j$ fix a metric $g_j$ with negative sectional curvature. Each $\EE_{g_j}$ is Morse-Bott, the nonconstant closed geodesics of $g_j$ are isolated and noncontractible, and there is a unique closed geodesic of $g_j$ in each nontrivial homotopy class of loops in $P_j$.
Let  $g$ be the metric on $L$ which is the sum of the $g_j$. Then, $\EE_g$ is also Morse-Bott, and the nonconstant $1$-periodic geodesics of $g$ occur in Morse-Bott nondegenerate critical submanifolds of dimension no greater than $1+ \dim P_2 + \dots +\dim P_k =1 +n - \dim P_1$. Hence, $d_g =  1+ n - \dim P_1$ and Theorem \ref{thm1} yields the desired class $A \in \ker i^*$.
\end{proof}

In the aspherical setting, Theorem \ref{splits} slightly improves the rigidity results for displaceable, split-hyperbolic Lagrangians from \cite{ks}, which yield a class in $\ker i^*$ 
with Maslov index in the interval  $[\dim P_1 -1, n+1\, (+1)]$.

\subsection{Maslov class rigidity and the fundamental selector}

Under the hypotheses of Theorem \ref{thm1}, one can also use the 
action selector corresponding to the fundamental class in Hamiltonian Floer homology (\cite{schw}),  to detect a class $B \in \ker i^*$ with $\masl(B) \in [n-d_g, n+1\,(+1)].$ This yields no information for the minimal Maslov number of Lagrangian tori as established in Theorem \ref{torus} and applied in Theorem \ref{klein}, since the lower bound for $\masl(B)$ in this case is zero. Moreover, unless one assumes that the Lagrangian submanifold is also monotone, it is not clear whether or not the classes $A$ and $B$ are distinct. Thus, the class $B$ need not yield  information which is complimentary to that obtained from the class $A$.

\begin{Remark}Roughly speaking, the bounds from \cite{ks} correspond to those implied by the class $B$. However, in the more general setting of \cite{ks}, one can not obtain these bounds by using action selectors. This would lead to bounds which only hold modulo the minimal Chern number of the ambient symplectic manifold.  The development of techniques which capture the same information as  the selector constructed here, and which work for more general classes of symplectic manifolds, will be the subject of future work.
\end{Remark}

\subsection{Maslov class rigidity and holomorphic discs}

Starting with  the work of Polterovich from \cite{po:a, po, po1},  many results on the rigidity of the Maslov class have also been obtained 
by considering moduli spaces of holomorphic discs with boundary on the Lagrangian submanifold.  These restrictions  are often as strong or stonger that those obtained here, but 
hold  for the special but important class of monotone displaceable Lagrangian submanifolds. Starting with the work of Oh in \cite{oh}, these results have been obtained primarily through various applications 
and refinements of Lagrangian
Floer homology, \cite{al, bi, bci, bco, bu, fooo}. Since the results presented here make no use of the monotonicity  assumption,  they can be viewed as complimentary to this 
body of work.

Recently, Fukaya has established a beautiful link between the moduli spaces of holomorphic discs with boundary on a Lagrangian submanifold and the 
string topology of the Lagrangian submanifold, \cite{fu}. Among several other applications, he has used this relationship to prove Audin's conjecture which asserts  that every 
Lagrangian torus in a symplectic vector space has minimal Maslov number equal to two. 
As described above, our present methods only allow us to reprove Viterbo's results concerning 
Lagrangian tori in symplectic vector spaces. The coupling of Fukaya's  methods with those developed  in \cite{ks} and here,
will also be the subject of future projects.

\subsection{Organization} In the next section we recall some basic definitions from the study of 
Hofer's geometry and recall a shortening result for certain Hamiltonian
paths. In Section 3, we describe some relevant tools from filtered Hamiltonian Floer theory which are then 
 used to define a new action selector in Section 4. In Section 5 we construct  
special Hamiltonian flows which are  supported near a Lagrangian submanifold.  
Our action selector is then applied to these flows to prove Theorem 
\ref{thm1} in Section 6.

\section{HAMILTONIAN FLOWS AND HOFER'S GEOMETRY}

Henceforth, we will assume that $(M, \om)$ is a closed symplectic manifold of dimension $2n$ which is symplectically  aspherical.
A smooth function $H$ on $S^1 \times M$ will be referred to as a {\bf Hamiltonian} on $M$. 
Here, we identify the circle $S^1$ with $\R / \Z$ and parameterize it with the coordinate $t \in [0,1]$.
Each Hamiltonian $H $ determines a $1$-periodic time-dependent 
Hamiltonian vector field $X_H$ via Hamilton's equation
\begin{equation*}
i_{X_H}\om = -dH_t,
\end{equation*}
where $H_t(\cdot) = H(t, \cdot)$.
The time-$t$ flow of $X_H$ will be denoted by $\phi^t_H$. It is defined for all $t \in [0,1]$ (in fact, for all $t \in \R$).

For  a Hamiltonian $H$ set
$$
H_{av} = \frac{1}{\Vol(M)} \int_0^1 \left( \int_M H_t \, \om^n \right) \, dt,
$$
where $\Vol(M)$ is the volume of $M$ with respect to $\om^n$.
Following \cite{ho1},  the Hofer length of the Hamiltonian path $\phi^t_H$ is defined to be
\begin{eqnarray*}
 \|H\| &=& \int_0^1 \max_M H_t \,\,dt-  \int_0^1 \min_M H_t \,\,dt\\
 {} &=& \left(\int_0^1 \max_M H_t \,dt - H_{av}\right) + \left( -\int_0^1 \min_M H_t \,\,dt + H_{av}\right)\\
{} &=& \|H\|^+ + \|H\|^-.
\end{eqnarray*}
The quantities $\|H\|^+$ and $\|H\|^-$ provide different measures of the length of $\phi^t_H$ 
called the positive and negative Hofer lengths, respectively. We will also consider 
the quantities $$\vvv H \vvv^+=\int_0^1 \max_M H_t \, dt  \,\,\,\,\,\text{  and  }\,\,\,\,\,  \vvv H \vvv^-=-\int_0^1 \min_M H_t \, dt$$ 
which need not be positive.

For a path of Hamiltonian diffeomorphisms $\psi_t$ with $\psi_0 = \id$, let $[\psi_t]$ be the class of Hamiltonian paths which are homotopic to $\psi_t$ relative to its endpoints. Denote the set of Hamiltonians which 
 generate the paths in $[\psi_t]$ by 
$$
C^{\infty}([\psi_t]) =\{ H \in C^{\infty}(S^1 \times M) \mid [\phi^t_H] = [\psi_t]\}.
$$  
The Hofer semi-norm of $[\psi_t]$ is then defined by 
$$
\rho([\psi_t]) = \inf_{H \in C^{\infty}([\psi_t])} \{ \|H\| \}.
$$
The positive and negative Hofer semi-norms of $[\psi_t]$ are defined similarly as
$$
\rho^{\pm}([\psi_t]) = \inf_{H \in C^{\infty}([\psi_t])} \{ \|H\|^{\pm} \}.
$$
Note, that  if $$\|H\|^{(\pm)} > \rho^{(\pm)}([\phi^t_H]),$$ 
then $\phi^t_H$  fails to minimize the (positive/negative) Hofer length in its homotopy class.

The displacement energy of a subset $U \subset (M,\om)$ is the quantity
$$
e(U,M,\om) = \inf_{\psi_t} \{ \rho([\psi_t]) \mid \psi_0=\id  \text{ and }\psi_1(U) \cap \overline{U} = \emptyset \},
$$
where $\overline{U}$ denotes the closure of $U$.
The following result relates the negative Hofer semi-norm and the displacement energy. It is a direct application of Sikorav's curve shortening technique and the reader is referred  to Lemma 4.2 of \cite{ke2} for the entirely similar proof of the analogous result for the positive Hofer semi-norm.

\begin{Proposition}\label{shorten}
Let $H \in C^{\infty}(M)$ be a time-independent  Hamiltonian that is constant and equal to its maximum value on the 
complement of an open set $U \subset M$. If $U$ has finite displacement energy 
and  if $\|H\|^- > 2 e(U) $, then $$\|H\|^- > \rho^-([\phi^t_H])  + \12 \|H\|^+.$$ 
In particular, the Hamiltonian path $\phi^t_H$ does not minimize the negative Hofer length in its homotopy class.
\end{Proposition}

\section{FLOER THEORY FOR SYMPLECTICALLY ASPHERICAL MANIFOLDS}

For any Hamiltonian $H$, we will denote the set of contractible $1$-periodic orbits of $\phi^t_H$ by $\PP(H)$. An element $x(t)$ of $\PP(H)$ is said to be nondegenerate 
if the linearized time-$1$ flow $d\phi^1_H \colon T_{x(0)}M \to T_{x(0)}M$ does not have one as an eigenvalue.
If every element of $\PP(H)$ is nondegenerate we will call $H$ a {\bf Floer Hamiltonian}.

The action of $x \in \PP(H)$ is given by
\begin{equation*}
\AC_H(x) = \int_0^1H(t, x(t)) \, dt - \int_{D^2} {\bf x}^* \om,
\end{equation*}
where ${\bf x} \colon D^2 \to M$ is a smooth map from the unit disc in $\C$ such that 
${\bf x}(e^{2\pi i t}) = x(t).$  This is well-defined since $\om|_{\pi_2(M)}$ is trivial. 
The action spectrum of $H$ is the set $$\SC(H) = \{ \AC_H(x) \mid x \in \PP(H)\}.$$

Following \cite{cz}, one can also define
the Conley-Zehnder index $\cz(x)$ of each $x$ in $\PP(H)$. Here we use the original normalization of this index.
In particular, for a local maximum $p$ of an  autonomous Morse Hamiltonian such that the Hessian of $H$ at $p$ 
is arbitrarily small (with respect to a fixed metric on $M$), we have  $\cz(p) = -n$.
 
\subsection{The Floer complex}  Let $H$ be a Floer Hamiltonian on $M$. 
For a generic $S^1$-family of $\om$-compatible almost complex structures $J_t = J$, 
 one can define the Floer complex 
$(\CF^*(H), \d_J)$ as follows. The cochain group $\CF^*(H)$ is the vector space over $\Z_2$  which is generated by the elements of $\PP(H)$ and is  graded
by the Conley-Zehnder index. The action of an element $\alpha = \sum n_j x_j$ of $\CF(H)$ is defined to be 
\begin{equation*}
\AC_H(\alpha) = \max \left\{ \AC_H(x_j) \mid n_j \neq 0 \right\}.
\end{equation*}

For $x$ and $y$ in $\PP(H)$, let $\MM(x,y;H,J)$ be the set of smooth maps  $w \colon \R \times S^1 \to M$
which satisfy  Floer's equation
\begin{equation*}
\p_s w + J(t,w)(\p_t w - X_H(t,w))=0,
\end{equation*} 
and have the following asymptotic behavior with respect to the smooth topology on
$C^{\infty}(S^1, M)$,
\begin{equation*}
 \lim_{s \to -\infty} w(s,t) = x(t)\,\,\,\, \text{  and  }\,\,\,\, \lim_{s \to +\infty} w(s,t) = y(t).
\end{equation*}
For a generic choice of $J$, each $\MM(x,y;H,J)$ is a smooth manifold of dimension
$\cz(y)-\cz(x)$. Moreover, $\R$ acts freely on  $\MM(x,y;H,J)$ by 
$ \zeta \cdot u(s,t) = u(\zeta +s,t)$.
The Floer coboundary map 
$$\d_J \colon \CF^*(H) \to \CF^{*+1}(H)$$ is then 
defined on generators of $\CF(H)$ by 
\begin{equation*}
\d_J(x) = \sum_{\cz(y) = \cz(x)+1} \#_2(\MM(x,y;H,J) / \R) \cdot y,
\end{equation*}
where $\#_2(\MM(x,y;H,J) / \R)$ is the number of elements in the zero dimensional manifold $\MM(x,y;H,J) / \R$, modulo two. 
Floer's gluing and compactness theorems imply that $\d_J \circ \d_J =0$, and we  denote the resulting cohomology $H^*(\CF(H), \d_J)$ by $\HF^*(H)$.

Before discussing the construction of Floer continuation maps and their role in identifying $\HF^*(H)$, we first recall a useful fact concerning  the Floer complexes for Hamiltonians
$H$ and $G$ which generate the same time one flow, i.e., $\phi^1_H = \phi^1_G$, and satisfy $H_{av} =G_{av}$. 
If  the Floer complex $(\CF^*(H), \d_J)$ is well-defined, then for 
\begin{equation}
\label{jtilde}
\widetilde{J} = d(\phi^t_H \circ (\phi^t_G)^{-1}) \circ J \circ d(\phi^t_G \circ (\phi^t_H)^{-1}) 
\end{equation}
 the Floer complex $(\CF^*(G), \d_{\widetilde{J}})$ is also  well-defined and is canonically isomorphic to $(\CF^*(H), \d_J)$.
 In particular, the map $x(t) \mapsto \widetilde{x}(t)=\phi^t_G\left( (\phi^t_H)^{-1} (x(t))\right)$ is a bijection from $\PP(H)$ to $\PP(G)$ and the 
 map $ u(s,t)  \mapsto \phi^t_G\left( (\phi^t_H)^{-1} (u(s,t))\right)$  is a bijection from $\MM(x,y;H,J)$ to  $\MM(\widetilde{x},\widetilde{y};G,\widetilde{J})$.
We will denote this isomorphism of complexes  by 
\begin{equation*}
(\CF^*(G), \d_{\widetilde{J}}) \equiv (\CF^*(H), \d_J).
\end{equation*}
It follows from the work of Seidel in \cite{se}, see also \cite{schw}, that 
it preserves both the action and Conley-Zehnder 
index of periodic orbits.

\subsection{Floer continuation maps} 
\label{identify}

Suppose that the Floer complexes of  two pairs $(G, J_G)$ and $(H,J_H)$  are well-defined. 
We now recall how Floer continuation maps define a natural isomomorphism between their 
respective Floer cohomology groups.
Consider a smooth  homotopy of data
$(F_s, J_s)$  which equals  $(G, J_G)$ for $s \ll 0$ and equals  $(H ,J_H)$ 
for $s \gg 0$.  Let $\MM_s(x,y;F_s, J_s)$ be the space of  maps $u \colon \R \times S^1\to M$ 
which satisfy 
\begin{equation*}
\p_su + J_s(t,u)(\p_t u - X_{F_s}(t,u)) = 0,
\end{equation*}
and have the following asymptotic behavior;
$$\lim_{s \to -\infty} u(s,t) = x(t) \in \PP(G)$$
and 
$$\lim_{s \to +\infty} u(s,t) = y(t) \in \PP(H).$$
For a generic choice of $(F_s, J_s)$  each $\MM_s(x,y; F_s, J_s)$ is a smooth manifold of dimension
$\cz(x)-\cz(y)$.  In this case we will say that the pair  $(F_s, J_s)$ is regular. The continuation map
$$\psi^G_H \colon \CF^*(G) \to \CF^*(H)$$
is then defined on generators by setting
\begin{equation*}
\psi^G_H(x) = \sum_{\cz(y) = \cz(x)} \#_2 \MM_s(x,y;F_s,J_s) \cdot y.
\end{equation*}
This is a chain map and the  induced map in cohomology, which we denote by $\Psi^G_H$,  is independent of the regular homotopy $(F_s, J_s)$.

For any  Floer Hamiltonians
$G$, $H$ and $K$ we have, at the level of  cohomology,
the factorization identity
\begin{equation}
\label{factor}
\Psi^H_G = \Psi^H_F \circ \Psi^F_G.
\end{equation}
This can be used to prove that each continuation map is an isomorphism and hence the Floer cohomology $\HF^*(H) = \H^*(\CF^*(H), \d_J)$ 
is independent of both $H$ and $J$.\footnote{Despite this fact, we use the notation $\HF(H)$ rather than $\HF$, in part to emphasize the role played by the 
elements of $\PP(H)$ in the definition of the Floer cohomology, but more importantly because of the dependence on $H$ of the filtered Floer homology.}

For small autonomous Hamiltonians we have the following normalization of the Floer cohomology.
\begin{Theorem}\cite{fhs,hs}
\label{morse}
There is  an open and dense set of smooth  Morse functions on $M$ such that for any $h$ in this set there is an  $\om$-compatible and $t$-independent almost complex structure $J$, and a positive  integer $m_0$, such that for all $m \geq m_0$ the Floer complex 
$$(\CF^*(h/m), \d_J)$$ is identical to $$(\CM_{n-*}(h/m), \p_{g_J}),$$ the Morse complex of $h/m$ with respect to the metric 
$g_J(\cdot, \cdot) = \om(\cdot, J \cdot)$. 
\end{Theorem}

Since the homology of a Morse complex is isomorphic to the singular homology of $M$, this immediately yields the following identification. 
\begin{Theorem}[\cite{fl1}]
\label{floer}
$$
\HF^*(H) = \H_{n-*}(M ; \Z_2).
$$
\end{Theorem}

\subsection{The fundamental class}
\label{fundamental}
One can define a fundamental class in Floer cohomology using Floer continuatuion maps and Theorem \ref{floer}. Let $\widehat{\FF}$ be the space of smooth Morse functions on $M$ which have exactly one local, and hence global, maximum. Choose a function  $h$ in $\widehat{\FF}$. Assuming that  $h$ is sufficiently small in the $C^2$-norm, the set of  orbits in $\PP(h)$ with Conley-Zehnder index $k$ coincides precisely with the 
set of critical points of $h$ with Morse index $n-k$. If $h$ attains its maximum value at 
$p \in M$, then Theorem \ref{floer} implies that the class $[p]$ is  nontrivial and generates the Floer cohomology in degree $-n$.
For any Floer Hamiltonian $H$, we then define the fundamental class in $\HF_{-n}(H)$ to be $[M] = \Psi^{h}_ H([p])$.

As is well known, the class $[M]$ is independent of the choice of the Morse function $h \in \widehat{\FF}$. We include a proof  of this fact as a model for similar future arguments. Let $f$ and $h$ be two 
functions in $\widehat{\FF}$ such that $\PP(f) = \Crit(f)$ and $\PP(h) = \Crit(h)$. Suppose that $f$ and $h$ attain their maximum values at  points $q$ and $p$
in $M$, respectively. Then for every Floer Hamiltonian $H$ and regular $J$ equation \eqref{factor} implies that
\begin{eqnarray*}
\Psi^f_H([q])  & = & \Psi^h_H \circ \Psi^f_h([q]) \\
{} & = & \Psi^h_H([p]),
\end{eqnarray*} 
as required.

\subsection{Filtered Floer cohomology}
\label{filter}
To define the action selectors of the next section we also require  tools from the filtered Floer cohomology theory
introduced  by Floer and Hofer in \cite{fh}. Consider a Floer Hamiltonian $H$ and two real numbers $a< b$ which lie outside of the 
action spectrum $\SC(H)$.  We  define  $\CF_{(a,b)}(H)$ to be the vector space over $\Z_2$ which is generated by the elements of 
$$
\PP_{(a,b)}(H) =\{ x \in \PP(H) \mid a < \AC_H(x)<b \}.
$$
For every map $w \in \MM(x,y;H,J)$ we have 
\begin{equation*}
\label{}
0 \leq \int_{\R \times S^1}  \om (\p_s w, J(w) \p_s w) \, dt \,ds = \AC_H(x) - \AC_H(y).
\end{equation*}
Hence, the Floer coboundary map $\d_J$ decreases actions, and it's restriction to $\CF_{(a,b)}(H)$ again satisfies 
$\d_J^2 =0.$  The filtered Floer cohomology of $H$ for the interval $(a, b)$ is then defined to be
$\HF^*_{(a,b)}(H) = H^*(\CF_{(a,b)}(H), \d_J).$ 

For three real numbers $a < b <c$ which lie outside $\SC(H)$,  $\CF_{(a,b)}(H)$ is  a subcomplex of $\CF_{(a,c)}(H)$ and 
$\CF_{(b,c)}(H)$ is naturally isomorphic to the quotient complex. This yields the following long exact sequence
\begin{equation*}
\label{}
\dots  \rightarrow \HF^{k-1}_{(b,c)}(H) \rightarrow \HF^k_{(a,b)}(H) \rightarrow \HF^k_{(a,c)}(H)\rightarrow \HF^k_{(b,c)}(H)\rightarrow  \HF^{k+1}_{(a,b)}(H) \rightarrow \dots.
\end{equation*}
We will set  $\HF_b(H) = \HF_{(-\infty, b)}(H)$, and will denote the inclusion map
in the exact sequence above by 
\begin{equation*}
i_b \colon  \HF^*_b (H) \to \HF^*(H).
\end{equation*}

As the notation suggests, $\HF_{(a,b)}(H)$ is independent of the choice of the 
family $J$ but depends on the Hamiltonian $H$. As we now recall, there are natural morphisms 
between filtered Floer cohomologies for different Hamiltonians. These maps rely on the useful fact that for
a map $u$ in $\MM_s(x,y;F_s,J_s)$ we have 
\begin{equation}
\label{key}
0 \leq \AC_G(x) - \AC_H(y) + \int_{\R \times S^1} \p_s F_s(s,t,u(s,t)) \, dt \,ds.
\end{equation}

If $G \geq H$ one can consider the class of homotopies $F_s$ from $G$ to $H$ such that $\p_sF_s \leq 0$. These are referred to as {\bf monotone homotopies}. Inequality \eqref{key} implies that a continuation map
$\psi^G_H$ determined by a monotone homotopy induces a chain map from $\CF^*_{(a,b)}(G) \to \CF^*_{(a,b)}(H)$. The map induced in cohomology,
\begin{equation*}
\Psi^G_H \colon \HF^*_{(a,b)}(G) \to \HF^*_{(a,b)}(H),
\end{equation*}
is independent of the choice of monotone homotopy.

More generally, for any two Hamiltonians $G$ and $H$ one can consider homotopies $F_s$ for which the quantity
 \begin{equation*}
\int_{\R \times S^1} \max_{p\in M} \p_s F_s(s,t,p) \, dt \,ds
\end{equation*}
is bounded from above by some constant $C>0$. Following \cite{gi:cois}, we call such homotopies  {\bf $C$-bounded}.
As in the case of monotone homotopies, inequality \eqref{key} implies that each regular homotopy $(F_s, J_s)$ for which $F_s$ is $C$-bounded homotopy 
determines a chain map from
$\CF^*_{(a,b)}(G) \to \CF^*_{(a+C, b+C)}(H)$ and hence a 
homomorphism
\begin{equation}
\label{c-bounded}
\Psi^G_H \colon \HF^*_{(a,b)}(G) \to \HF^*_{(a+C, b+C)}(H).
\end{equation}
Again, this map is independent on the choice of $C$-bounded homotopy, \cite{gi:cois}. 

\begin{Remark}To avoid overly cumbersome notation
we will use  $\Psi_\ast^\ast$ to denote all continuation homomorphisms even when they are defined by special classes of 
$C$-bounded homotopies. This latter fact will be implied, when necessary,  by the inclusion of the domain and target of the homomorphism,
as in equation \eqref{c-bounded}.
\end{Remark}

\begin{Example}
\label{linear}
For any two Floer Hamiltonians $G$ and $H$ consider a \emph{linear homotopy} of  the 
form $$F_s = (1-b(s))G + b(s) H,$$ where $b \colon \R \to [0,1]$ is a nondecreasing function which
equals zero for $s \leq -1$ and equals one for $s\geq 1$.  Note that  
\begin{eqnarray*}
\int_{\R \times S^1} \max_{p\in M} \p_s F_s(s,t,p) \, dt \,ds & = & \int_{\R \times S^1} \max_{p\in M} \dot{b}(s)(H-G)(s,t,p) \, dt \,ds \\
{} & = & \vvv H-G\vvv^+.
\end{eqnarray*}
Hence $F_s$ is a $(\vvv H-G \vvv^+)$-bounded homotopy which can be used to define a homomorphism
\begin{equation}
\label{looky}
\Psi^G_H \colon \HF^*_{(a,b)}(G) \to \HF^*_{(a+\vvv H-G\vvv^+, b+\vvv H-G \vvv  ^+)}(H).
\end{equation}
\end{Example}

Henceforth we will denote a chain map defined using a linear homotopy of Hamiltonians from $G$ to $H$ by $\overline{\psi}^G_H$.
The corresponding map in (filtered) Floer cohomology will be denoted by $\overline{\Psi}^G_H$. 

\begin{Remark}
\label{regular}
To achieve regularity, it might be necessary to perturb 
the precise linear homotopy defined above away from its endpoints. However, any sufficiently small perturbation yields the same 
map in cohomology, and so the notation  $\overline{\Psi}^G_H$ is unambiguous.  Since we are working coefficients in $\Z/2\Z$, the same is true of the chain level
map  $\overline{\psi}^G_H$.  
\end{Remark}

Let $F_s$ be a $C$-bounded homotopy from $G$ to $H$, and let $F'_s$ be a $C'$-bounded homotopy from $H$ to $K$. For sufficiently large $R>0$
$$
F_s \odot_R F'_s = 
\begin{cases}
  F_{s+R}    & \text{for} \,\,s \leq 0, \\
  F'_{s-R}    & \text{otherwise}.
\end{cases}
$$
is a smooth homotopy from $G$ to $K$ which is  $(C+C')$-bounded.

\begin{Lemma}(\cite{gi:cois})
\label{splice}
For sufficiently large $R>0$, the map $$\Psi^G_K \colon \HF_{(a,b)}(G) \to \HF_{(a+C+C', b+C+C')}(K)$$ induced 
by $F_s \odot_R F'_s$ is equivalent to the composition $\Psi^G_H \circ \Psi^H_K$ whose factors are induced by 
$F_s$ and $F'_s.$
\end{Lemma}

We now define a more subtle fundamental class in $\HF_a(H)$ for sufficiently large $a$.
This will be a key part of our definition of a new action selector.

\begin{Proposition}\label{middle}
For every  $a > \rho^+([\phi^t_H]) + H_{av}$ which lies outside of $\SC(H)$, there is a well-defined and nontrivial class $[M]_a$ in $\HF_a^{-n}(H)$ 
with the following properties 
\renewcommand{\theenumi}{\roman{enumi}}
\renewcommand{\labelenumi}{(\theenumi)} 
\begin{enumerate}
\item $i_a([M]_a) = [M]$;
\item If $H$ is a Floer Hamiltonian and  $J$ is regular,  then $[M]_a$ is represented by a cycle 
$\alpha$ in $\CF_a^{-n}(H)$  with  $\AC_H(\alpha) \leq \rho^+([\phi^t_H]) + H_{av}$.
\end{enumerate}
\end{Proposition}

\begin{proof}

Set 
$$
\Delta^+_H = \min \{ c \in \SC(H) \mid c > \rho^+([\phi^t_H]) + H_{av} \}.
$$
Choose  a Hamiltonian $G$ in $C^{\infty}([\phi^t_H])$ such that  $G_{av} = H_{av}$
and 
\begin{equation}
\label{gbound}
\vvv G \vvv^+ < \min\{a, \Delta^+_H\}.
\end{equation}  Let 
 $f$ be a Morse function in $\widehat{\FF}$ which satisfies  
\begin{equation}
\label{fbound}
2 \| f \| < \min\{ (a - \vvv G \vvv^+), (\Delta^+_H -\vvv G \vvv^+)\}.
\end{equation}
Set  $$\alpha' = \overline{\psi}^{f}_{G}(p),$$
where $p$ is the unique local maximum of $f$. 
Since we are working only with $\Z_2$-coefficients, 
the class $\alpha'$ is uniquely determined by $f$. 
Moreover, it follows from Example \ref{linear} together with \eqref{gbound} and \eqref{fbound} that
\begin{eqnarray*}
\AC_G(\alpha') & \leq & f(p) +\vvv G-f \vvv^+ \\
{} & \leq & \vvv G \vvv^+ + \|f\| \\
{} & < & \min\{a, \Delta^+_H\}. 
\end{eqnarray*}

Since $\phi^1_G=\phi^1_H$ and $G_{av} = H_{av}$,  we have $\SC(G) =\SC(H)$ and so 
\begin{equation*}
\label{ }
\AC_G(\alpha') \leq \rho^+( [\phi^t_H]) +H_{av}.
\end{equation*}
In fact, for all $a$ we have 
\begin{equation*}
\label{identa}
(\CF^*_a(G), \d_{\widetilde{J}}) \equiv (\CF^*_a(H), \d_J).
\end{equation*}
where $\widetilde{J}$ is defined as in \eqref{jtilde}.
Because $\overline{\psi}^{f}_{G}(p)$ belongs to $\CF^{-n}_{a}(G)$, 
 we can therefore identify $\alpha'$ with a  class $\alpha$ in $\CF_a^{-n}(H)$ 
such that 
\begin{equation*}
\AC_H(\alpha) = \AC_G(\alpha') \leq \rho^+([\phi^t_H]) + H_{av}.
\end{equation*}
The definition of the fundamental class then implies that  for $a > \rho^+([\phi^t_H]) + H_{av}$
$$i_a([\alpha']) = [\overline{\psi}^{f}_{G}(p)] = [M] \in \HF(G),$$ and hence 
$$i_a([\alpha]) =[M] \in \HF(H).$$
Setting $$[M]_a = [\alpha],$$ it only remains to show that this class is well-defined. 
Arguing as for the fundamental class $[M]$, one can show that $[\alpha']$ and hence $[\alpha]$ does not depend on 
the choice of the Morse function $f$ in $\widehat{\FF}$. The proof that $[M]_a$ is independent of the choice of $G$ is 
more involved.

Let $F$ be another Hamiltonian in  $C^{\infty}([\phi^t_H])$ which satisfies  $F_{av} = H_{av}$
and $\vvv F \vvv^+ < \min\{a, \Delta^+_H\}.$ Choose a Morse function $f \in \widehat{\FF}$ such that 
$$2\|f\| < \min\left\{ (a -\vvv G \vvv^+), (\Delta^+_H -\vvv G \vvv^+),(a - \vvv F \vvv^+), (\Delta^+_H -\vvv F \vvv^+)\right\}.$$ 
If $\alpha_F$ is the class in $\CF^{-n}_{a}(G)$ which corresponds to $\overline{\psi}^{f}_{F}(p)$
under the identification 
\begin{equation}
\label{identfg}
(\CF_a(F), \d_{\widetilde{J^G}}) \equiv (\CF_a(G), \d_{J^G}),
\end{equation}
then it suffices for us to show that $[\alpha_F] = \left[ \overline{\psi}^{f}_{G}(p)\right]$. With the Morse function
$f$ fixed as above, we will actually prove that 
\begin{equation}
\label{step}
\alpha_F = \overline{\psi}^{f}_{G}(p).
\end{equation}

Let $x$ be any element in $\PP(G)$ with Conley-Zehnder index equal to $-n$. Set
$\varrho_t = \phi^t_{F} \circ (\phi^t_G)^{-1}$.
In the identification \eqref{identfg}, $x$ corresponds to  the orbit
$y(t) = \varrho_t(x(t))$ in $\PP(F)$. To verify \eqref{step}, we need to prove that 
\begin{equation}
\label{moduli}
\#_2\MM_s(p,x; G_s, J^G_s)
=
\#_2\MM_s(p,y;F_s, J^F_s),
\end{equation}
where $G_s$ and $F_s$ are small perturbations of the linear homotopies from 
$f$ to $G$ and  from $f$ to $F$ described in Example \ref{linear} (see Remark \ref{regular}).
Here $J^G_s$ and $J^F_s$ are regular families of almost complex structures.  
The choice of these families does not effect the count, modulo two, of the moduli spaces in \eqref{moduli} which they help define.

Since both $F$ and $G$ belong to $C^{\infty}([\phi^t_H])$ we can choose a family of functions $H_s$ in $C^{\infty}([\phi^t_H])$
such that $H_s = G$ for $s \leq -1$ and $H_s = F$ for $s \geq 1$. Define $\widetilde{F}_s$ by 
$$
\widetilde{F}_s = 
\left\{
  \begin{array}{lll}
G_s & \hbox{for $s \leq 1$;} \\
    H_{s-2}, & \hbox{for $1 \leq s \leq 3 $;}\\
F, & \hbox{for $s \geq 3 $.}
   \end{array}
\right.
$$
Consider the family of contractible Hamiltonian loops 
$$\varrho_{s,t}= \phi^t_{\widetilde{F}_s} \circ (\phi^t_{G_s})^{-1},$$
and let
$$\widetilde{J}_s= d \varrho_{s,t} \circ J^G_s \circ d (\varrho_{s,t}^{-1}).$$
For each value of $s$, $\varrho_{s,t}$ is a loop based at the identity,
and 
$$
\varrho_{s,t}= 
\left\{
  \begin{array}{ll}
    \id, & \hbox{for $s \leq 1$;} \\
    \phi^t_{H_{s-2}} \circ (\phi^t_{G})^{-1}, & \hbox{for $1 \leq s \leq 3$;} \\
    \varrho_t, & \hbox{for $s \geq 3$.}
   \end{array}
\right.
$$

From $\varrho_{s,t}$ we obtain two families of normalized Hamiltonians, $A_s$ and $B_s$, defined by  
$$ \p_s (\varrho_{s,t}(p)) = X_{A_s}(\varrho_{s,t}(p))$$
and
$$ \p_t (\varrho_{s,t}(p)) = X_{B_s}(\varrho_{s,t}(p)).$$
The standard composition formula for Hamiltonian flows implies that
$${B_s}= \widetilde{F}_s -G_s \circ \varrho^{-1}_{s,t},$$
where $(G_s \circ \varrho^{-1}_{s,t}) (t,p) = G_s(t,\varrho^{-1}_{s,t}(p))$. 

To prove \eqref{moduli}, we first note that the image of $\MM_s(p,x; G_s, J^G_s)$ under the map
$$u(s,t) \mapsto \varrho_{s,t}( u(s,t))$$
is the space $\MM_s(p,y; \widetilde{F}_s, A_s, \widetilde{J}_s)$ of smooth maps $v \colon \R \times S^1 \to M$ which satisfy the equation
\begin{equation}
\label{utwist}
\partial_s v- X_{A_s}(v)+\widetilde{J}_s(v)(\partial_t v - X_{\widetilde{F}_s}(v))=0,
\end{equation}
and have the following asymptotic behavior:
$$\lim_{s \to -\infty} v(s,t) = p
\,\,\text{  and  }
\lim_{s \to +\infty} v(s,t) = y(t).$$

To see this, note that the desired limiting behavior of $v = \varrho_{s,t}(u) \in \MM_s(p,y; \widetilde{F}_s, A_s, \widetilde{J}_s)$ is determined 
by that of $u$. The fact that $v$ satisfies equation \eqref{utwist} then follows from the following simple computation
\begin{eqnarray*}
  \p_s v + \widetilde{J}_s(v)\p_t v &=& d(\varrho_{s,t}) \p_su +
  X_{A_s}(v)
   + \widetilde{J}_s( v)\left( d(\varrho_{s,t}) \p_tu + X_{B_s}(v)\right) \\
  {} &=& d(\varrho_{s,t})\left( \p_su + J_s(u)\p_tu\right) + X_{A_s}(v)+  \widetilde{J}_s(v)X_{B_s}(v)\\
  {} &=& d(\varrho_{s,t})J_s(u)X_{G_s}(u) + X_{A_s}(v)+  \widetilde{J}_s(v)X_{B_s}(v) \\
  {} &=&  \widetilde{J}_s(v)\left( d(\varrho_{s,t})X_{G_s}(u) \right) + X_{A_s}(v)+  \widetilde{J}_s(v)X_{B_s}(v) \\
  {} &=&  \widetilde{J}_s(v)X_{G_s \circ \varrho_{s,t}^{-1}}(v) + X_{A_s}(v)+  \widetilde{J}_s(v)X_{B_s}(v)\\
  {} &=& X_{A_s}(v) + \widetilde{J}_s(v)X_{\widetilde{F}_s}(v).
\end{eqnarray*}

For $\lambda \in \R$, let $(\widetilde{F}^{\lambda}_s, A^{\lambda}_s, \widetilde{J}^{\lambda}_s)$ be  a triple of smooth families  of functions and $\om$-compatible almost complex structures such that 
\begin{itemize}
  \item $(\widetilde{F}^{\lambda}_s, A^{\lambda}_s, \widetilde{J}^{\lambda}_s) = (\widetilde{F}_s, A_s, \widetilde{J}_s)$
  for $\lambda \leq -1$;
  \item $(\widetilde{F}^{\lambda}_s, A^{\lambda}_s, \widetilde{J}^{\lambda}_s) = (F_s, 0,J^F_s)$
  for $\lambda \geq 1$;
  \item For some $\tau>0$, $A^{\lambda}_s =0$ for all $|s| \geq \tau$ and $\lambda \in \R$.
\end{itemize}

Let $\WW(p,y;\widetilde{F}^{\lambda}_s, A^{\lambda}_s, \widetilde{J}^{\lambda}_s )$ be the space of maps $v \colon \R \times S^1 \to M$
such that 
\begin{equation*}
\label{}
\partial_s v- X_{A^{\lambda}_s}(v)+\widetilde{J}^{\lambda}_s(v)(\partial_t v - X_{\widetilde{F}^{\lambda}_s}(v))=0,
\end{equation*}
$$\lim_{s \to -\infty} v(s,t) = p$$
and 
$$\lim_{s \to +\infty} v(s,t) = y(t).$$

For a generic choice of the family $\widetilde{J}^{\lambda}_s$, the space $\WW(p,y;\widetilde{F}^{\lambda}_s, A^{\lambda}_s, \widetilde{J}^{\lambda}_s )$ is a smooth cobordism between $\MM_s(p,y; \widetilde{F}_s, A_s, \widetilde{J}_s)$ and $\MM_s(p,y; F_s, 0, J^F_s)= \MM_s(p,y; F_s, J^F_s)$. Hence,
$$
\#_2\MM_s(p,x; G_s, J^G_s)= \#_2\MM_s(p,y; \widetilde{F}_s, A_s, \widetilde{J}_s)= \#_2\MM_s(p,y; F_s, J^F_s),
$$
as required.

\end{proof}

\begin{Remark}
\label{simple}
For $a> \vvv H \vvv^+$, we can choose $G=H$ in the construction of $[M]_a$. In this case, 
\begin{equation*}
 [M]_a =\overline{ \Psi}^f_H ([p])
\end{equation*}
for any function $f \in \widehat{\FF}$ with $2 \|f\| < a- \vvv H \vvv^+.$
\end{Remark}

\section{AN ACTION SELECTOR FOR PINNED HAMILTONIANS}



%


We now define a new action selector $\widehat{\sigma}$ for a  special, but useful, class of Hamiltonians.
\begin{Definition}
We say that a Hamiltonian $H$ is {\bf pinned} if there is a point $Q\in M$ such that  
\begin{itemize}
\item  For all $(t,p) \in [0,1] \times M$ we have $$ H(t,Q) \geq H(t,p),$$ with equality only along $[0,1] \times \{Q\}$;
 \item The Hessian of $H_t$ at $Q$ is nondegenerate for all $t \in [0,1]$, and the linearized flow $d \phi^t_H \colon T_Q M \to T_QM$  
 has no nonconstant periodic orbits with period less than or equal to one.
 \end{itemize}
 \end{Definition}
 
Pinned Hamiltonians arise naturally in the study of the length minimizing properties of Hamiltonian paths with respect to Hofer lengths. In particular, slightly weaker versions of these two properties are necessary  for a Hamiltonian to generate a positive length minimizing path in its homotopy class, \cite{bp, lm, ust}.

In what follows, the point $Q$ at which a pinned Hamiltonian attains its maximum values will be specified.
In this case, we say that $H$ is {\bf pinned at $Q$}.

 \begin{Definition} 
Let $H$ be a Hamiltonian  that is pinned at $Q$. A Morse function $f$ on $M$  is said to {\bf dominate}  $H$ if
 \begin{itemize}
  \item The maximum value of $f$ is zero and is only achieved at  $Q$;
  \item  The function 
 \begin{equation*}
f_H (t,p) \eqdef f(p) + H(t,Q).
\end{equation*} 
is greater than or equal to $H$, with equality only along $[0,1] \times \{Q\}$.
\item  There are no nonconstant $1$-periodic orbits of the Hamiltonian flow of $f$, i.e. $\PP(f) = \Crit(f)$.
\end{itemize}
\end{Definition}
Every pinned Hamiltonian $H$ is dominated by some Morse function. In particular, if $f$ is a Morse function with a 
unique local maximum at $Q$, then $\eps f$ dominates $H$ for small enough $\eps > 0$.

At the chain level we have the following  result.

\begin{Lemma}
\label{down}
Let $H$ be a Floer Hamiltonian which is pinned at $Q$ and let $f$ be a Morse function which dominates $H$.
Then for the Floer continuation map defined by a monotone linear homotopy from $f_H$ to $H$ we have 
 $$\overline{\psi}^{f_H}_H(Q) = Q +\beta$$ for some $\beta$ in $\CF_{-n}(H)$ with
\begin{equation*}
\AC_H(\beta) < \AC_H(Q) = \vvv H \vvv^+.
\end{equation*}
\end{Lemma}

\begin{proof}
The bounds on the action 
of $\beta$ follow from inequality \eqref{key} and the fact that our homotopy  from $f_H$ to $H$ is monotone.
If one presumes that the linear homotopy $F_s = (1-b(s))f_H + b(s)H$ is part of a regular set of continuation data
$(F_s,J_s)$, then by \eqref{key} the only element of $\MM_s(Q,Q;F_s, J_s)$ is the constant map $u(s,t)=Q$
and we are done. Without this presumption, one must  prove  that the 
number of elements in $\MM_s(Q,Q;F'_s, J_s)$ is equal to one modulo two, 
where $F'_s$ is a generic monotone perturbation of $F_s$ for which $(F'_s, J_s)$ is regular. This can be established  as in Lemma 2.5 of \cite{gi:con}.

\end{proof}

We can now define another version of the fundamental class for the filtered Floer cohomology of pinned Hamiltonians. 

\begin{Proposition}
\label{in}
Let $H$ be a Hamiltonian which is pinned at $Q$. For every $a > \vvv H\vvv^+$ which lies outside of $\SC(H)$  there is a well-defined and nontrivial class $[\widehat{M}]_a$ in $\HF_a^{-n}(H)$ with the following properties
\renewcommand{\theenumi}{\roman{enumi}}
\renewcommand{\labelenumi}{(\theenumi)} 
\begin{enumerate}
\item $i_a ([\widehat{M}]_a) = [M]$;
\item If $H$ is a Floer Hamiltonian and $J$ is regular, then $[\widehat{M}]_a$ is represented by a cycle $\gamma=Q +\beta$ in $\CF_a^{-n}(H)$
such that $\AC_H(\beta)  < \vvv H\vvv^+ .$
\end{enumerate}
\end{Proposition}

\begin{proof}
By the continuity of the filtered Floer homology we may assume the $H$ is a Floer Hamiltonian.
For
$
\Delta^+_a= \min\{ c \in \SC(H) \mid c > a\},
$
let $f$ be a dominating function for $H$ such that $2\|f\| < \Delta^+_a-a$ and set
\begin{equation*}
[\widehat{M}]_a= \overline{\Psi}^{f_H}_H([Q]). 
\end{equation*}  
Clearly, $\Crit(f)$ is equal to $\PP(f_H)$.  As well, any regular homotopy from $(f,J')$ to $(H,J)$ induces a regular homotopy from $(f_H,J')$ to $(H,J)$ such that 
the corresponding maps $\psi^{f}_H$ and $\psi^{f_H}_H$ are identical. Hence,  $\Psi^{f_H}_H([Q])= [M]$,  for any 
 dominating function $f$.  This implies the first property of $[\widehat{M}]_a$. The second property follows immediately from Lemma \ref{down}.
 
 It only remains to show that $[\widehat{M}]_a$ does not depend on the choice of $f$. Let $h$ be another 
 Morse function which dominates $H$ and satisfies $2\|h\| < (\Delta^+_a-a)$. It follows from Lemma \ref{splice} and our choice of $\Delta^+_a$ that
\begin{equation*}
\label{}
\overline{\Psi}^{f_H}_H = \overline{ \Psi}^{h_H}_H \circ \overline{\Psi}^{f_H}_{h_H}   \colon  \HF_{a}(f_H) \to \HF_{a +\vvv f_H -h_H\vvv^+}(H) = \HF_a(H).
\end{equation*}
Hence
\begin{eqnarray*}
\overline{\Psi}^{f_H}_H([Q]) & = &\overline{ \Psi}^{h_H}_H \circ \overline{\Psi}^{f_H}_{h_H}([Q]) \\
{} & = & \overline{ \Psi}^{h_H}_H([Q]),
\end{eqnarray*}
as required.
\end{proof}

We now consider the class of Hamiltonians which are pinned and whose flows do not minimize the 
positive Hofer norm. Set
\begin{equation*}
\HH(Q) = \{ H \in C^{\infty}(S^1 \times M) \mid H \text{ is pinned at $Q$}, \, \|H\|^+ > \rho^+([\phi^t_H])\}.
\end{equation*}
Let $H$ be a Floer Hamiltonian in $\HH(Q)$ and consider an $a > \vvv H \vvv^+$ which lies outside of $\SC(H)$.  By Proposition \ref{middle}, for any regular $J$ there is a
cycle  $\alpha$ in $\CF_a^{-n}(H)$ such that $[\alpha] = [M]_a$
and
\begin{equation}
\label{action ineq1}
\AC_H(\alpha) \leq \rho^+([\phi^t_H]) +H_{av} < \vvv H \vvv^+.
\end{equation}
Proposition \ref{in} implies that, again for a regular $J$, there is a cycle 
$\gamma =Q+\beta$ in $\CF_a^{-n}(H)$ such that 
\begin{equation*}
i_a([\gamma]) =[M]
\end{equation*}
and
\begin{equation}
\label{action ineq2}
\AC_H(\beta) \leq \rho^+([\phi^t_H]) +H_{av} < \vvv H \vvv^+ = \AC_H(Q).
\end{equation}
If $a$ is sufficiently large, for example if $a$ is greater than $\max\{\AC_H(x) \mid x \in \PP(H) \}$, then we have $$[\alpha] = [\gamma] \in \HF_a(H).$$
However, for values of $a$ just slightly larger than $ \vvv H \vvv^+$ we have $$[\alpha] \neq [\gamma] \in \HF_a(H).$$ To see this,
note that the action inequalities \eqref{action ineq1}  and  \eqref{action ineq2} imply that  $Q$ appears with coefficient one in $\gamma - \alpha = Q+\beta-\alpha$. Since $H$ is a Floer Hamiltonian and $\AC_H(Q) = \vvv H \vvv^+$, for values of $a$ just slightly larger than $ \vvv H \vvv^+$,  the orbit $Q$ is not in the image of $\d_J \colon \CF_a(H) \to \CF_a(H).$ In particular, for values of $a$ close enough to $ \vvv H \vvv^+$ there are no elements 
of $\PP(H)$ which have action in the interval  $( \vvv H \vvv^+ , a]$. 
Hence,  $[Q+\beta-\alpha]$
is not trivial in $\HF_a(H)$ and $[\alpha] \neq [\gamma]$. 

This change in behavior as $a$ approaches $ \vvv H \vvv^+$ from infinity, leads one to the following  definition. For Floer Hamiltonian $H$ in $\HH(Q)$ set
\begin{equation*}
\widehat{\sigma}(H) = \inf \left\{  a >  \vvv H \vvv^+ \mid   [\widehat{M}]_a =  [M]_a \in \HF_a(H) \right\}.
\end{equation*}
\begin{Remark}
By the discussion above, we have $\widehat{\sigma}(H) > \vvv H \vvv^+$. Hence, $\widehat{\sigma}(H) $ is not equal to any of the action selectors defined  by Schwarz in \cite{schw} for closed, symplectically aspherical symplectic manifolds. 
\end{Remark}
\begin{Lemma}
\label{realize2}
If $H \in \HH(Q)$ is a Floer Hamiltonian, then there is an element $\widehat{x}$ in $\PP(H)$ such that 
$\cz(\widehat{x}) = -n-1$ and $\AC_H(\widehat{x})=\widehat{\sigma}(H).$
 \end{Lemma}

\begin{proof}
As described above, for every $a > \widehat{\sigma}(H)$ the class $[Q+\beta-\alpha]$ is trivial in
$\HF_a(H)$. Hence, there is a class $\eta$ in $\CF_a(H)$ such that  $\AC_H(\eta) \geq \widehat{\sigma}(H)$ and $\d_J(\eta) = Q+\beta -\alpha$.
In particular, for all $a > \widehat{\sigma}(H)$ there is an orbit  in $\PP(H)$ whose  Conley-Zehnder index is $-n-1$
and whose action is in the interval $[ \widehat{\sigma}(H), a)$. Since there are only finitely many elements of $\PP(H)$, the result follows.
\end{proof}

We now establish a  continuity property for the action selector $\widehat{\sigma}$. 
\begin{Proposition}
\label{contin}
If $H$ and $G$ are Floer Hamiltonians in $\HH(Q)$, then 
\begin{equation}
\label{cont}
| \widehat{\sigma}(H)- \widehat{\sigma}(G) | < \| H-G \| +|H_{av} - G_{av}|.
\end{equation}
\end{Proposition}
\begin{proof}
We first consider the case of two Hamiltonians $H$ and $G$ with $H_{av} = G_{av}$. Let $f$  be a dominating function for both  $H$ and $G$ and fix 
an $a>\widehat{\sigma}(H)$. The definition of $\widehat{\sigma}(H)$ implies that 
\begin{equation}
\label{ha}
[\widehat{M}]_a  = [M]_a \in \HF_a(H).
\end{equation}
Since $\widehat{\sigma}(H)>  \vvv H \vvv ^+$, it follows from Remark \ref{simple}  that for a dominating function $f$ with $\|f\|$ sufficiently small, equation 
\eqref{ha} is equivalent to
\begin{equation}
\label{ass}
\overline{\Psi}^{f_H}_H([Q]) = \overline{\Psi}^f_H([Q])  \in \HF_a(H).
\end{equation}

The properties of  $C$-bounded homotopies  discussed in Section \ref{filter} imply
that 
\begin{equation}
\label{splitsy}
\overline{\Psi}^{f_G}_G =  \overline{\Psi}^H_G \circ \overline{\Psi}^{f_H}_H   \circ \overline{\Psi}^{f_G}_{f_H} \colon  \HF_{a}(f_G) \to \HF_{a +\|G -H\|}(G).
\end{equation}
In particular, $\overline{\Psi}^{f_G}_G$  is defined by a monotone homotopy and hence is $C$-bounded for any $C>0$. 
By Lemma \ref{splice} and Example \ref{linear},  the map $\overline{\Psi}^H_G \circ \overline{\Psi}^{f_H}_H   \circ \overline{\Psi}^{f_G}_{f_H}$ is equal to
the Floer continuation map defined by a  homotopy which is $(\vvv G-H\vvv^+ + \vvv f_H -f_G\vvv^+)$-bounded. Since 
\begin{eqnarray*}
\vvv G-H\vvv^+ + \vvv f_H -f_G\vvv^+ & = &\vvv G-H\vvv^+ + \vvv H\vvv^+  - \vvv G\vvv^+ \\
{} & \leq & \vvv G-H\vvv^+ + \vvv H - G\vvv^+  \\
{} & = & \|G-H\|,
\end{eqnarray*}
it follows that $\overline{\Psi}^H_G \circ \overline{\Psi}^{f_H}_H   \circ \overline{\Psi}^{f_G}_{f_H}$ is also equal to
the map from $\HF_{a}(f_G)$ to $\HF_{a +\|G -H\|}(G)$ defined by a  $(\|G-H\|)$-bounded homotopy. This yields equation \eqref{splitsy}.

Together, equations \eqref{ass} and \eqref{splitsy} imply that
in $\HF_{a +\|G -H\|}(G)$ we have
\begin{eqnarray*}
\overline{\Psi}^{f_G}_G([Q]) &=& \overline{\Psi}^H_G \circ \overline{\Psi}^{f_H}_H   \circ \overline{\Psi}^{f_G}_{f_H} ([Q])\\
{} & = & \overline{\Psi}^H_G \circ\overline{ \Psi}^{f_H}_H ([Q])\\
{} & = & \overline{\Psi}^f_G  ([Q])
\end{eqnarray*}
Since $a > \widehat{\sigma}(H)$ and $H_{av} =G_{av}$  we have $a + \|G-H\| > \vvv H \vvv^+ + \vvv G-H \vvv^+ \geq \vvv G \vvv^+$. Hence, by Remark 
\ref{simple} again,  the equality $ \overline{\Psi}^{f_G}_G([Q])=\overline{\Psi}^f_G  ([Q])$ in $\HF_{a +\|G -H\|}(G)$
is equivalent to
 \begin{equation*}
[\widehat{M}]_{a+ \|G-H\|}  = [M]_{a+ \|G-H\|} \in \HF_{a+ \|G-H\|}(G).
\end{equation*}
In other words, for every $H$ and $G$ with $H_{av} =G_{av}$ and every $a > \widehat{\sigma}(H)$, we have  shown that $a +\|G-H\| >\widehat{\sigma}(G)$. In the case, $H_{av} =G_{av}$ we then have 
$
\widehat{\sigma}(G) \leq \widehat{\sigma}(H) +\|G-H\|
$
and hence
\begin{equation}
\label{first}
| \widehat{\sigma}(H)- \widehat{\sigma}(G) | \leq \|H-G\|.
\end{equation}

For any Hamiltonians $H$ and $G$,  inequality \eqref{first} implies
\begin{eqnarray*}
| \widehat{\sigma}(H)- \widehat{\sigma}(G) | & = & | \widehat{\sigma}(H-H_{av})- \widehat{\sigma}(G-G_{av}) +H_{av} -G_{av} | \\
{} & \leq  & \| (H-H_{av})- (G-G_{av}) \| +|H_{av} -G_{av} |\\
{} & = &  \| H- G\| +|H_{av} -G_{av} |, 
\end{eqnarray*}
as desired.
\end{proof}
Proposition \ref{contin} implies that $\widehat{\sigma}$ is continuous with respect to the $C^0$-norm on $C^{\infty}(S^1 \times M)$.  This allows us to define $\widehat{\sigma}(H)$ for every $H \in \HH(Q)$.

\section{A SPECIAL HAMILTONIAN FLOW NEAR $L$}

In this section we describe the construction of  a Hamiltonian $\hl$ whose flow is supported near $L$ and does not minimize the 
negative Hofer length in its homotopy class. The nonconstant contractible periodic orbits of this Hamiltonian are related directly to closed  geodesics of a  metric $g$ on $L$.
One aspect of this relation is in terms of familiar indices; the Conley-Zehnder index of a $1$-periodic orbit  of $\hl$ is related to the Morse index of the 
corresponding closed geodesic and the Maslov index of the class in $\ker i^*$ that it represents.  This index relation is recalled at the end of the section 
 as Proposition \ref{index}.

\subsection{Geodesic flows}
\label{geodesic}
Let $g$ be a Riemannian metric on a closed manifold $L$.  As noted in the introduction, the energy functional of $g$
on the space of smooth loops $C^{\infty}(S^1, L)$,  is defined by
$$
\EE_g(q(t)) = \int_0^1 \12  |\dot{q}(t)|^2   \, dt 
$$
and the critical points of $\EE_g$ are the closed geodesics
of $g$ with period equal to one. The closed geodesics of $g$ 
with any  positive period $T>0$ correspond to the $1$-periodic orbits of the metric $\frac{1}{T} g$, and
are thus the critical points of the functional $\EE_{\frac{1}{T}g}$. 
 
Denote the Hessian of $\EE_g$ at a critical point $q(t)$, by $Hess(\EE_g)_q$.
The dimension of the space on which $Hess(\EE_g)_q$ is negative definite
is  finite. This dimension is the Morse index of $q$ and will be denoted by $\mor(q)$. 
The Hessian of $\EE_g$ at a critical point also  has a finite dimensional kernel which is 
always nontrivial (unless $L$ is a point). 

A submanifold $D \subset C^{\infty}(S^1,L)$ which consists of critical points of
$\EE_g$ is said to be  Morse-Bott nondegenerate if the dimension of the kernel of 
$Hess(\EE_g)_q$ is equal to the dimension of $D$ for every $q \in D$. 
The energy functional $\EE_g$ is said to be Morse-Bott if all the $1$-periodic
geodesics are contained in Morse-Bott nondegenerate critical submanifolds of $\EE_g$.
Note that if $\EE_g$ is Morse-Bott, then so is $\EE_{\frac{1}{T}g}$ for any $T>0$.

\subsection{Geodesic flows supported near a Lagrangian submanifold}
\label{sec:knu}
Let  $L$ be a displaceable Lagrangian submanifold of a symplectic manifold $(M, \om)$.  Any Riemannian metric $g$ on $L$ determines a cometric on $T^*L$ which we may use to define the neighborhood of the zero section  
$$
U_r=  \{ (q,p) \in T^*L  \mid |p| < r\}.
$$
Weinstein's neighborhood theorem implies that, for sufficiently small $r>0$, there is a neighborhood 
of $L$ in $(M, \om)$ which is symplectomorphic to $U_r$ where $T^*M$ is equipped with the symplectic form 
$d\theta$ where $\theta$ is the Liouville one-form. We fix a value of $r$ 
for which this holds, and will henceforth  identify $U_r$ with a 
neighborhood of $L$ in $(M, \om)$.  Decreasing this value of $r$, if necessary, we may also assume that
$U_r$ has finite displacement energy. For a subinterval $I \subset [0,r)$, we will use the notation
$$
U_I = \{ (q,p) \in U_r \mid |p| \in I \}.
$$

For a numbers  $\eps \in  (0,r/8)$ and $C > 0$, let $\nu = \nu_{\scriptscriptstyle{\eps,C}} \colon [0,+\infty) \to \R$ be a smooth  function with the following properties
\begin{itemize}
  \item $\nu  = 0$ on $[0, \eps]$;
  \item $\nu', \nu'' > 0$ on $(\eps, 2\eps)$; 
  \item $\nu  = -\eps + Cs $ on $[2\eps, r-2\eps]$;
   \item $\nu' > 0$ and $\nu''<0 $ on $(r-2\eps, r-\eps)$; 
  \item $\nu = A < rC$ on $[r-\eps ,+\infty)$.
\end{itemize}
For future reference, we note that these properties imply that the function $s \mapsto -\nu(s)+\nu'(s)s $
is nondecreasing, and takes values in the interval $[0,\eps]$ for $s \in [0, r-2\eps]$.

Define the Hamiltonian $K_{\nu}$ on $M$ by
$$K_{\nu}(q,p)=
\begin{cases}
 \nu(|p|) & \text{if $(q,p)$ is in $U_r$}, \\
 A    & \text{otherwise}.
\end{cases}
$$
The Hamiltonian flow of $K_{\nu}$ is trivial on $U_{\eps}$ and the complement of $U_{r-\eps}$.
Elsewhere,  
\begin{equation*}
X_{K_{\nu}} (q,p) = \left( \frac{\nu'(|p|)}{|p|}  \right)  X_{K_g}(q,p),
\end{equation*}
where $K_g = \12 |p|^2$ is the kinetic energy Hamiltonian  which generates the 
cogeodesic flow of $g$ on $(T^*L, d \theta)$. In particular,  if $x(t)$ is a 
nonconstant $1$-periodic orbit of $K_{\nu}$, then $x(t)$ can be expressed in standard $(q,p)$
coordinates on $T^*L$ as $x(t)=(q(t),p(t))$, where $q(t)$ is a closed geodesic of $g$ on $L$
with period  $|p(0)| / \nu'(|p(0)|)$ and length
\begin{equation*}
\int_0^1 | \dot{q}(t)| \, dt = \int_0^1 \nu'(|p(t)|) \, dt = \nu'(|p(0)|).
\end{equation*}
Hence, for every nonconstant orbit $x(t)=(q(t),p(t))$ in $\PP(K_{\nu})$ we have the uniform bound
\begin{equation}
\label{uniform bound}
|\dot{q}(t)| \leq C.
\end{equation}
Moreover, if $C$ is not the length of a closed geodesic of $g$ on $L$, then the nonconstant orbits of $K_{\nu}$ occur on the level sets  
contained in $U_{(\eps , 2\eps)}$ or  $U_{(r-2\eps , r-\eps)}$,
where $\nu$ is convex or concave, respectively.

Note that
\begin{eqnarray*}
\|K_{\nu}\|^- & = & -\left(\min K_{\nu} - \int_M K_{\nu} \, \om^n \Big/ \Vol(M)\right)\\
{} & = & \int_M K_{\nu} \, \om^n \Big/ \Vol(M)  \\
{} & > & A \left(\frac{\Vol(M) - \Vol(U_r)}{\Vol(M)}\right).
\end{eqnarray*}
Let 
\begin{equation*}
E = \frac{4e(U_r)}{r} \left(\frac{\Vol(M)}{\Vol(M) - \Vol(U_r)}\right).
\end{equation*}
By construction, $A> \frac{Cr}{2}$, so for $C>E$ we have
\begin{equation}
\label{too big}
\|K_{\nu}\|^- > 2e(U_r).
\end{equation}
In this case, Proposition \ref{shorten} implies that the Hamiltonian flow of $K_{\nu}$
does not minimize the negative Hofer length. We will henceforth assume that  $C>E$. 

\subsection{A perturbation of $K_{\nu}$}

The following  result  describes a perturbation of $K_{\nu}$ which is a Floer Hamiltonian whose constant periodic orbits 
 capture the Morse homology of $M$ and whose nonconstant periodic orbits retain many of the properties 
of those of $K_{\nu}$.  It follows from the proof of Proposition 3.5 in \cite{ks}.
\begin{Proposition}
\label{function}
Let $L$ be a Lagrangian submanifold of $(M,\om)$ and let $g$ be a metric on $L$ whose energy  functional is Morse-Bott. 
Let  $\nu= \nu_{\scriptscriptstyle{\eps,C}}$ be a function as above,  with $C>E$. There is an $\eps_0>0$ such that for 
every $0<\eps<\eps_0$ there is a Floer Hamiltonian  
$\hl = \hl^{\scriptscriptstyle{\eps,C}}$ which has the following properties:
  \newcounter{Lcount}
  \begin{list}{{\bf (H\arabic{Lcount})}}
    {\usecounter{Lcount}
    \setlength{\rightmargin}{\leftmargin}}
 
\item Each $\hl = \hl^{\scriptscriptstyle{\eps,C}}$ is within a distance $\eps$  of $K_{\nu}=K_{\nu_{\scriptscriptstyle{\eps,C}}}$ in the $C^{\infty}$-topology on $C^{\infty}(S^1 \times M)$. In particular, we may assume that the $C^0$-distance between them is small enough so that 
\begin{equation}
\label{ }
\|\hl^{\scriptscriptstyle{\eps,C}} - K_{\nu_{\scriptscriptstyle{\eps,C}}}\| + |(\hl^{\scriptscriptstyle{\eps,C}})_{av} - (K_{\nu_{\scriptscriptstyle{\eps,C}}})_{av}| < \epsilon.
\end{equation}
\item The constant $1$-periodic orbits $\hl$ correspond to the critical points of 
  a Morse function $F$ on $M$. Near these points the Hamiltonian flows of $\hl$
  and $c_0 F$ are identical for some arbitrarily small constant $c_0>0$;
\item  The nonconstant $1$-periodic orbits of $\hl$ are contained in $U_{(\eps , r-\eps)}$ and if $C$ is not the length 
  of a closed geodesic of $g$ on $L$, then they are contained in $U_{(\eps , 2\eps)} \cup U_{(r-2\eps , r-\eps)}$. In either case, 
  for every nonconstant   $x(t)=(q(t), p(t))$ in $\PP(\hl)$ we have  the uniform  bound 
 \begin{equation}
\label{uniform bound2}
 |\dot{q}(t)| < 2C;
\end{equation}
  \item There is a point $Q \in L \subset M$ which is the unique local minimum
  of $\hl(t, \cdot)$ for all $t \in [0,1]$. 
\item The flow $\phi^t_{\hl}$ does not minimize the negative Hofer length in its homotopy class.
\end{list}
\end{Proposition}

The construction of $\hl$ from \cite{ks}  also implies that each nonconstant  contractible $1$-periodic orbit $x(t)$ of $\hl$ projects to 
a nondegenerate  critical point $q(t)$ of a perturbed energy
functional of the form 
$$
\EE_{\frac{1}{T}g,V}(q) = \int_0^1 \left( \frac{1}{2T^2}  |\dot{q}(t)|^2 - V(t, q(t)) \right) \, dt.
$$
Here,  $T>0$ and $V \colon S^1 \times L \to \R$ is smooth and arbitrarily small with respect to the $C^\infty$-metric. Thus, the 
bound \eqref{uniform bound2} follows immediately from \eqref{uniform bound}. The perturbed geodesic $q$ can be associated to a 
unique critical submanifold $D$ of $\EE_{\frac{1}{T}g}$ and  the Morse index of $q$ belongs to $[\mor(D), \mor(D) + \dim D].$ 

The following index relation follows from the work of Duistermaat \cite{du}, Weber \cite{we}, and Viterbo \cite{vi}. For a
proof the reader is referred to Proposition 4.4 of \cite{ks}.

\begin{Proposition}
\label{index}
Let $L$ be a Lagrangian submanifold of $(M,\om)$ and let $g$ be a metric on $L$ which has nonpositive sectional curvature and whose energy  functional $\EE_g$ is Morse-Bott. Denote by $d_g$ 
the maximum dimension of a critical submanifold of $\EE_g$ which consists of nonconstant 
closed geodesics.
Let $x(t) = (q(t), p(t))$ be a contractible $1$-periodic orbit of a Hamiltonian $\hl$ as in Proposition \ref{function}.
If $x$ is contained in $U_{(\eps,2\eps)}$, then 
\begin{equation*}
\cz(x) -1-d_g  \leq \masl([q]) \leq \cz(x) (+1),
\end{equation*}
where the $(+1)$ term contributes only if $q^*TM$ is not orientable.
\end{Proposition}

\section{PROOF OF THEOREM \ref{thm1}}

First we recall that for a Hamiltonian $H$
the flow of the Hamiltonian  
\begin{equation*}
\widetilde{H}(t,p) = -H(t, \phi^t_H(p))
\end{equation*}
satisfies  
$$\phi^t_{\widetilde{H}} =(\phi^t_H)^{-1}.$$
Clearly, $$ \|\widetilde{H}\|^{\pm} = \|H\|^{\mp}.$$
Moreover,  the map 
\begin{equation}
\label{inverse}
x(t) \mapsto  \phi^t_{\widetilde{H}}(x(0))
\end{equation}
is a bijection from $\PP(H)$ onto $\PP(\widetilde{H})$ such that
\begin{equation*}
\cz\left(\phi^t_{\widetilde{H}}\left(x(0)\right)\right) = -\cz(x)
\end{equation*}
and 
\begin{equation*}
\AC_{\widetilde{H}}\left(\phi^t_{\widetilde{H}}(x(0))\right) = -\AC_H(x).
\end{equation*}

Let  $\nu_{\scriptscriptstyle{\eps,C}}$ be a family of smooth functions defined as in Section \ref{sec:knu}, which depend smoothly on the parameters $\eps$ and $C$.
Fix a corresponding smooth family of Hamiltonians $\hl^{\eps,C}$, as in Proposition \ref{function}, 
for $(\eps,C) \in (0, \eps_0) \times (E, \infty)$. Set $$\hlo^{\scriptscriptstyle{\eps,C}}(t,p) =-\hl^{\scriptscriptstyle{\eps,C}}(t, \phi^t_{\hl^{\scriptscriptstyle{\eps,C}}}(p)).$$ By properties ${\bf (H2)}$ and  ${\bf (H4)}$, each $\hlo^{\scriptscriptstyle{\eps,C}}$ is pinned at $Q$. By property ${\bf(H5)}$, $\phi^t_{\hl^{\eps,C}}$ does not minimize the negative Hofer length. Hence, the function
$$\widehat{\sigma}(\eps,C) \eqdef \widehat{\sigma}(\hlo^{\scriptscriptstyle{\eps,C}})$$ is well-defined on $(0, \eps_0)\times (E, \infty)$. By Proposition \ref{contin}, this function is continuous.
The following result relates  $\widehat{\sigma}(\eps,C)$ to $1$-periodic orbits of the Hamiltonians $\hlo^{\scriptscriptstyle{\eps,C}}$.

\begin{Lemma}
For each $\eps$ in  $(0,  \eps_0)$ there is a 
nonconstant periodic orbit $\widehat{x}_{\eps}$ in $\PP(\hlo^{\scriptscriptstyle{\eps,C}})$ such that 
\begin{equation}
\label{b}
\AC_{\hlo^{\scriptscriptstyle{\eps,C}}}(\widehat{x}_{\eps}) = \widehat{\sigma}(\eps,C),
\end{equation}
and 
\begin{equation}
\label{c}
\cz(\widehat{x}_{\eps})=-n- 1.
\end{equation}
\end{Lemma}

\begin{proof}
Lemma  \ref{realize2} implies  the existence of an orbit $\widehat{x}_{\eps}$ 
with properties \eqref{b}, and \eqref{c}. By condition {\bf (H2)}, every constant $1$-periodic orbit 
of $\hlo^{\scriptscriptstyle{\eps,C}}$ has Conley-Zehnder index no less than $-n$, and so \eqref{c} implies that $\widehat{x}_{\eps}$ must be nonconstant. 
\end{proof}

\begin{Lemma}
\label{extend}
The function $\widehat{\sigma}(\eps,C)$  has a continuous extensions to $[0, \eps_0) \times (E, \infty)$.
\end{Lemma}

\begin{proof}

We begin by showing that the limit as $\eps \to 0$ of   $\widehat{\sigma}(\eps, C)$ exists for each $C$.
For any sequence $\eps_j$ in $(0, \eps_0)$, Proposition \ref{contin} and {\bf(H1)} imply that
\begin{eqnarray*}
|\widehat{\sigma}(\eps_k,C) -\widehat{ \sigma}(\eps_l,C)| & \leq & \| \hlo^{\scriptscriptstyle{\eps_k,C}}- \hlo^{\scriptscriptstyle{\eps_l,C}}\|+
| (\hlo^{\scriptscriptstyle{\eps_k,C}})_{av}- (\hlo^{\scriptscriptstyle{\eps_l,C}})_{av}|\\
{} & \leq & \eps_k +\eps_l + \| K_{\nu_{\eps_k,C}} - K_{\nu_{\eps_l,C}}\| + | (K_{\nu_{\eps_k,C}})_{av} - (K_{\nu_{\eps_l,C}})_{av}|.
\end{eqnarray*}
Hence, if $\eps_j \to 0$, then the sequence $\widehat{\sigma}(\eps_j,C)$ is Cauchy and converges. It suffices to show that the limit of this sequence is independent of the choice of  $\eps_j \to 0$. 

For the orbits $\widehat{x}_{\eps} = (\widehat{q}_{\eps} , \widehat{p}_{\eps} )$, it follows from  \eqref{uniform bound2} that $\|\dot{\widehat{q}}_{\eps}(t)\|<2C$ for all $t \in [0,1]$ and all $\eps \in (0, \eps_0)$. By the Arzela-Ascoli theorem, for every sequence $\eps_j \to 0$ there is a subsequence $\eps_{j_k}$ such that the $\widehat{q}_{\eps_{j_k}}$ converge to a closed geodesic $\widehat{q}_{0,C}$ of $g$. The corresponding lifts $\widehat{x}_{\eps_{j_k}} = (\widehat{q}_{\eps_{j_k}}, \widehat{p}_{\eps_{j_k}})$ converge to some $(\widehat{q}_{0,C},\widehat{p}_{0,C})$
and hence
\begin{equation}
\label{ident}
\lim _{j \to \infty}\widehat{\sigma}(\eps_j,C) = \lim_{k \to \infty }\AC_{\hlo^{\scriptscriptstyle{\eps_{j_k},C}}}(\widehat{x}_{\eps_{j_k}}). 
\end{equation}
For small $\eps >0$, the fact that $\hl^{\scriptscriptstyle{\eps,C}}$ is within $\eps$ of $K_{\nu_{\scriptscriptstyle{\eps,C}}}$ in the $C^{\infty}$-topology, implies that 
\begin{eqnarray*}
\AC_{\hlo^{\scriptscriptstyle{\eps,C}}}(\widehat{x}_{\eps}) & = & \AC_{-K_{\nu_{\scriptscriptstyle{\eps,C}}}}\left( \widehat{x}_{\eps}\right) + o(\eps)\\
{} & = &-\int_0^1 K_{\nu_{\scriptscriptstyle{\eps,C}}}(\widehat{x}_{\eps}(t)) \, dt - \int_{D^2} {\bf \widehat{x}_{\eps}}^* \om + o(\eps)\\
{} & = & -\left(\nu_{\scriptscriptstyle{\eps,C}}\right)(|\widehat{p}_{\eps}|) -
              \left( \int_0^1 \widehat{p}_{\eps} \dot{\widehat{q}_{\eps}} \, dt + \int_{D^2}{\bf \widehat{q}_{\eps}}^*\om \right)+o(\eps)\\
{} & = & -\left(\nu_{\scriptscriptstyle{\eps,C}}\right)(|\widehat{p}_{\eps}|)+ \left(\nu_{\scriptscriptstyle{\eps,C}}\right)'(|\widehat{p}_{\eps}|) |\widehat{p}_{\eps}|  
-\int_{D^2} {\bf \widehat{q}_{\eps}}^*\om +o(\eps).
\end{eqnarray*}
As well, the defining properties of the $\nu_{\scriptscriptstyle{\eps,C}}$ imply that each function  $$s \mapsto -\left(\nu_{\scriptscriptstyle{\eps,C}}\right)(s)+ \left(\nu_{\scriptscriptstyle{\eps,C}}\right)'(s)s $$
takes values in $[0, \eps]$ for $s \in [0, r-2\eps]$. Thus,
$\lim_{k \to \infty }\AC_{\hlo^{\scriptscriptstyle{\eps_{j_k},C}}}(\widehat{x}_{\eps_{j_k}})$ is equal to either  $$- \int_{D^2} {\bf \widehat{q}_{0,C}}^* \om$$ or  $$-Cr +\ell(\widehat{q}_{0,C})r - \int_{D^2} {\bf \widehat{q}_{0,C}}^* \om.$$  
Now, the union of the quantities $\ell(\widehat{q}_{0,C})$  and $\int_{D^2} {\bf \widehat{q}_{0,C}}^* \om$,  over all closed geodesics $\widehat{q}_{0,C}$, is a countable set. By \eqref{ident},  this implies that $\lim_{j \to \infty} \widehat{\sigma}(\eps_j,C)$ is independent of the sequence $\eps_j \to 0$. Hence, $\lim_{\eps \to 0} \widehat{\sigma}(\eps,C)$ exists for each $C \in (E, \infty)$. We denote this limit by $\widehat{\sigma}(0,C)$.

It remains to show that $\widehat{\sigma}(0,C)$ is continuous in $C$. This is an immediate consequence of the following inequalities
\begin{eqnarray*}
|\widehat{\sigma}(0,C) -\widehat{\sigma}(0,C')| & = & \lim_{\eps \to 0}|\widehat{\sigma}(\eps,C) -\widehat{\sigma}(\eps,C')| \\
{} & \leq & \lim_{\eps \to  0} \left(\left\| \hlo^{\scriptscriptstyle{\eps,C}}- \hlo^{\scriptscriptstyle{\eps,C'}}\right\| +
\left| (\hlo^{\scriptscriptstyle{\eps,C}})_{av}- (\hlo^{\scriptscriptstyle{\eps,C'}})_{av} \right| \right) \\
{} & \leq &  2r|C-C'|.
\end{eqnarray*}

\end{proof}

If  $C$ is not the length of a closed geodesic of $g$ on $L$, then property {\bf (H3)} implies that the nonconstant orbits $\widehat{x}_{\eps}$  lie in either $U_{(\eps , 2\eps)}$ or  $U_{(r-2\eps, r-\eps)}$.  

\begin{Proposition}
\label{locate}
If $C$ is  sufficiently large and is not the length of a closed geodesic of $g$ on $L$, then for some $\eps \in (0,\eps_0)$ the orbit $ \widehat{x}_{\eps}$ lies in $U_{(\eps , 2\eps)}$. \end{Proposition}
\begin{proof}
Assume that $C$ is not the length of a closed geodesic of $g$ on $L$ and that $\widehat{x}_{\eps} = (\widehat{q}_{\eps} , \widehat{p}_{\eps} )$ belongs to $U_{(r-2\eps, r-\eps)}$ for all $\eps \in (0, \eps_0)$.  It follows from the proof of Lemma \ref{extend} that 
\begin{equation*}
\widehat{\sigma}(0,C) = -Cr +\ell(\widehat{q}_{0,C})r - \int_{D^2} {\bf \widehat{q}_{0,C}}^* \om, 
\end{equation*}
for some closed geodesic $\widehat{q}_{0,C}$ of $g$. Since both $\ell(\widehat{q}_{0,C})$ and $\int_{D^2} {\bf \widehat{q}_{0,C}}^* \om$
take values in a countable set, and $\widehat{\sigma}(0,C)$ is a continuous function of $C$, we have  
\begin{equation}
\label{Down}
\lim_{\eps \to 0} \widehat{\sigma}(\eps,C) = -Cr +K_r
\end{equation}
for some constant $K_r$.
For suffuciently large $C>0$, this contradicts the fact that 
\begin{eqnarray*}
\widehat{\sigma}(0,C) & \geq  &  \lim_{\eps \to 0} \vvv\hlo^{\scriptscriptstyle{\eps,C}}\vvv^+ \\
{} & = & \vvv-K_{\nu_{\scriptscriptstyle{0,C}}} \vvv^+\\
{} & = & 0.
\end{eqnarray*}

\end{proof}

Proposition \ref{locate}  implies that for a sufficiently large choice of $C>0$ which lies outside the length spectrum of the metric $g$, there is an $\eps \in (0,\eps_0)$ and
a  periodic orbits $\widehat{x}_{\eps} \in \PP(\hlo^{\scriptscriptstyle{\eps,C}})$ such that $\widehat{x}_{\eps}$ has Conley-Zehnder index $-n-1$ and is contained in 
$U_{(\eps,2\eps)}$. The orbit $\widehat{x}_{\eps}$ corresponds to a periodic orbit $x_{\eps} = (q_{\eps} , p_{\eps})$ in  $\PP(\hl^{\scriptscriptstyle{\eps,C}})$ which has Conley-Zehnder index $n+1$ and 
still lies in $U_{(\eps,2\eps)}$. It then follows from Proposition \ref{index} that
\begin{equation*}
n +1- d_g  \leq \masl([q_{\eps}]) \leq n+1 \,(+1).
\end{equation*}

Let $A$  be the class $[q_{\eps}]$. The proof of Theorem \ref{thm1} will be complete if we can show that $\omega(A)>0$. By definition, we have 
\begin{equation}
\label{start}
\AC_{\hlo^{\scriptscriptstyle{\eps,C}}}(\widehat{x}_{\eps}) = \widehat{\sigma}(\eps,C) > \vvv \hlo^{\scriptscriptstyle{\eps,C}} \vvv^+.
\end{equation}
Arguing as above, we may assume that for an arbitrarily small $\eps >0$ the orbit $\widehat{x}_{\eps}$ lies in $U_{(\eps,2\eps)}$.
In this  case \eqref{start} implies that 
\begin{equation*}
\label{ }
-\int_{D^2} {\bf \widehat{q}_{\eps}}^*\om +o(\eps) \geq \vvv \widetilde{K}_{\nu_{\scriptscriptstyle{\eps,C}}}\vvv^+ +o(\eps),
\end{equation*}
and we have 
\begin{equation}
\label{end}
\omega(A) = \int_{D^2} {\bf q_{\eps}}^*\om \geq 0.
\end{equation}
To obtain the strict inequality we invoke another observation from \cite{vi}, (see page 301). Assume that for every class $C \in \ker i^*$ with $\masl(C) \in [n+1-d_g, n+1\,(+1)]$ we have 
$\omega(C) \leq 0$.  Then there is an isotopic Lagrangian submanifold $L'$ arbitrarily close to $L$, and hence also displaceable, such  that every class $C \in \ker (i')^* =\ker i^*$ with $\masl(C) \in [n+1-d_g, n+1\,(+1)]$ satisfies  $\omega(C)<0$. This contradicts the existence, established above, of the class $A$ satisfying $\omega(A)\geq 0$ and  $\masl(A) \in [n+1-d_g, n+1\,(+1)].$

\end{document}